\documentclass[12pt]{amsart}
\usepackage{amssymb,amscd,amsthm,amsmath,color}
\usepackage{fullpage}
\newcommand{\PP}{\mathbb{P}}
\newcommand{\OO}{\mathcal{O}}

\newcommand{\CC}{\mathbb{C}}

\newcommand{\xX}{\mathcal{X}}
\newcommand{\yY}{\mathcal{Y}}
\newcommand{\ver}{\operatorname{vert}}

\theoremstyle{plain}
\newtheorem{lemma}{Lemma}[section]
\newtheorem*{theorem*}{Theorem}
\newtheorem*{lemma*}{Lemma}
\newtheorem*{proposition*}{Proposition}
\newtheorem*{conjecture*}{Conjecture}
\newtheorem*{corollary*}{Corollary}
\newtheorem*{problem*}{Problem}
\newtheorem{theorem}[lemma]{Theorem}
\newtheorem{conjecture}[lemma]{Conjecture}
\newtheorem{corollary}[lemma]{Corollary}
\newtheorem{proposition}[lemma]{Proposition}

\newtheorem{question}[lemma]{Question}

\theoremstyle{definition}

\newtheorem{remark}[lemma]{Remark}

\begin{document}

\title{Algebraic hyperbolicity of the very general quintic surface in $\PP^3$}
\author[I. Coskun]{Izzet Coskun}
\address{Department of Mathematics, Statistics and CS \\University of Illinois at Chicago, Chicago, IL 60607}
\email{coskun@math.uic.edu}

\author[E. Riedl]{Eric Riedl}
\email{ebriedl@uic.edu}

\subjclass[2010]{Primary: 32Q45, 14H50. Secondary: 14J70, 14J29}
\keywords{Algebraic hyperbolicity, hypersurfaces, genus bounds}
\thanks{During the preparation of this article the first author was partially supported by the  NSF grant DMS-1500031 and NSF FRG grant  DMS 1664296 and  the second author was partially supported  by the NSF RTG grant DMS-1246844.}

\begin{abstract}
We prove that a curve of degree $dk$ on a very general surface of degree $d \geq 5$ in $\PP^3$ has geometric genus at least $\frac{dk(d-5)+k}{2} + 1$. This improves bounds given by G. Xu. As a corollary, we conclude that the very general quintic surface in $\PP^3$ is algebraically hyperbolic.  
\end{abstract}

\maketitle

\section{Introduction}
A complex projective variety $X$ is {\em algebraically hyperbolic} if there exists an $\epsilon > 0$ such that for any curve $C \subset X$ of geometric genus $g(C)$ we have $$2g(C)-2 \geq \epsilon \deg(C).$$ In particular, algebraically hyperbolic varieties contain no rational or elliptic curves. Algebraic hyperbolicity is intimately related to metric notions of hyperbolicity. Recall that a variety $X$ is {\em Brody hyperbolic} if it admits no holomorphic maps from $\CC$. By Brody's Theorem \cite{Brody}, Brody hyperbolicity is equivalent to the nondegeneracy of the Kobayashi metric on compact manifolds. Brody hyperbolicity is conjectured to control many geometric and arithmetic properties of $X$. For example, Demailly \cite{Demailly} proves that Brody hyperbolic varieties are algebraically hyperbolic and conjectures the following.

\begin{conjecture}
A smooth complex projective variety $X$ is Brody hyperbolic if and only if it is algebraically hyperbolic.
\end{conjecture}

Hypersurfaces provide a natural testing ground for deep conjectures on hyperbolicity. Hence, we are led to the following question.

\begin{question}
\label{ques-algHyp}
For which $n$ and $d$ is a very general hypersurface of degree $d$ in $\PP^n$ algebraically hyperbolic?
\end{question}

More generally, one can ask the following.

\begin{question}
\label{ques-whichCurves}
For a very general hypersurface of degree $d$ in $\PP^n$, for which pairs $(e,g)$ does there exist a  curve of degree $e$ and  geometric genus $g$?
\end{question}

The study of the hyperbolicity of very general hypersurfaces in $\PP^n$ has a long history (see \cite{CoskunHyperbolicity,Demailly, DemaillyElGoul, Mcquillan, Siu}). For example, Question \ref{ques-algHyp} has been resolved in many cases. Results of Voisin \cite{Voisin, Voisincorrection} prove that if $n \geq 4$, then $$2g-2 \geq e(d-2n+2).$$  Therefore, a very general hypersurface of degree $d \geq 2n-1$ in $\PP^n$ is algebraically hyperbolic for $n \geq 4$. In the case $n=3$, any curve on $X$ is a complete intersection of type $(d,k)$. Geng Xu \cite{Xu} imporved results of Ein \cite{Ein} and proved that $$g \geq \frac{dk(d-5)}{2} + 2,$$ showing that surfaces in $\PP^3$ of degree at least $6$ are algebraically hyperbolic. However, despite all this interest, the case of quintics in $\PP^3$ has remained open for the past 20 years (see \cite{Demailly, Demaillynew}). Our goal in this paper is to prove the following.

\begin{theorem}
\label{thm-algHyp}
Let $X \subset \PP^3$ be a very general surface of degree $d \geq 5$. Then the geometric genus of any curve in $X$ of degree $dk$ is at least $\frac{dk(d-5)+k}{2} + 1$.
\end{theorem}

In particular, when $d=5$, the bound on the genus of a curve of degree $5k$ specializes to 
$$2 g(C) - 2 \geq k = \frac{1}{5} \deg(C).$$ We obtain the following corollary.

\begin{corollary}
A very general quintic surface in $\PP^3$ is algebraically hyperbolic.
\end{corollary}

\subsection*{Organization of the paper} In \S \ref{sec-Prelim}, we recall the basic setup developed by Ein, Voisin, Pacienza, Clemens and Ran. In \S \ref{sec-Proof}, we prove Theorem \ref{thm-algHyp}. 

\subsection*{Acknowledgments} We would like to thank Lawrence Ein, Mihai P\u{a}un and Matthew Woolf for useful discussions.

\section{preliminaries} \label{sec-Prelim}

In this section, we review the basic setup due to \cite{ClemensRan, Ein, Ein2, Pacienza, Pacienza2, Voisin, Voisincorrection} and collect facts relevant to the rest of this paper. The main goal of this section is to prove Corollary \ref{cor-mapToNormalBundle}. We always work over the complex numbers $\CC$. 

Let $S_d = H^0(\PP^n, \OO_{\PP^n}(d))$.  Suppose that a general hypersurface of degree $d$ in $\PP^n$ admits a generically injective map from a smooth projective variety of deformation class $Y$. After an \'{e}tale base change $U \to S_d$, we can find a generically injective map $$h:\yY \to \xX$$ over $U$, where $\xX$ is the universal hypersurface of degree $d$ in $\PP^n$ over $U$ and $\yY$ is a smooth family of pointed varieties over $U$ with members in the deformation class of $Y$. We choose $\yY$ so that the codimension of $h(\yY)$ in $\xX$ is $n-1-\dim Y$. 
Let  $$\pi_1: \xX \to U \quad \mbox{and} \quad \pi_2: \xX \to \PP^n$$ denote the two natural projections of the universal hypersurface over $U$. Let the {\em vertical tangent sheaf} $T_{\xX}^{\ver}$ be defined by the natural sequence
$$0 \rightarrow T_{\xX}^{\ver} \rightarrow T_{\xX} \rightarrow \pi_2^* T_{\PP^n} \rightarrow 0.$$ 

Since every hypersurface contains a variety of deformation class $Y$, the family $\yY$ dominates $U$ under $\pi_1 \circ h$.  Moreover, without loss of generality, we may assume that $\yY$ is stable under the $GL_{n+1}$ action on $\PP^n$, so $\pi_2 \circ h$ dominates $\PP^n$ \cite{Pacienza, Voisin}. Furthermore, the invariance under $GL_{n+1}$ also implies that the map  $T_{\yY} \rightarrow h^* (\pi_2^* T_{\PP^n})$ is surjective. Let $T_{\yY}^{\ver}$ denote the kernel
$$0 \rightarrow T_{\yY}^{\ver} \rightarrow T_{\yY} \rightarrow h^* (\pi_2^* T_{\PP^n}) \rightarrow 0.$$

Let $t \in U$ be a general closed point.  Let $X_t$ denote the corresponding fiber of $\xX$. Let $$h_t: Y_t \to X_t$$ denote the corresponding map from the fiber $Y_t$ of $\yY$ over $t$ to $X_t$. Let $$i_t: X_t \rightarrow \xX \quad \mbox{and} \quad j_t: Y_t \rightarrow \yY$$ denote the inclusion of $X_t$ in $\xX$ and $Y_t$ in $\yY$, respectively. Observe that $i_t \circ h_t = h \circ j_t$. Let $N_{h/\xX}$ denote the normal sheaf to the map $h: \yY \to \xX$ and let $N_{h_t/X_t}$ denote the normal sheaf to $h_t$ defined as the cokernels of the following natural sequences 
$$ 0 \to T_{\yY} \to h^*T_{\xX} \to N_{h/\xX} \to 0$$
$$0 \rightarrow T_{Y_t} \rightarrow h_t^* T_{X_t} \rightarrow N_{h_t/X_t} \rightarrow 0.$$ 

Following Ein, Voisin, Pacienza, Clemens and Ran, it is standard to relate  $N_{h_t/X_t}$  to sheaves arising from the family $\yY \to \xX$. Unfortunately, the prior work does not explicitly state the theorem we will use. For the reader's convenience, we reprove the main statements, emphasizing the key points from our perspective. We will make repeated use of the following lemma. 

\begin{lemma}
\label{lem-staysExact}
Let $\phi: Z_0 \rightarrow Z$ be a morphism of projective varieties.
Let
\begin{equation}\label{seq-lemstaysexact}
0 \to E \to F \to G \to 0 
\end{equation}
be a short exact sequence of sheaves on $Z$, where $E$ and $F$ are vector bundles.  If $\phi^* E \to \phi^* F$ is generically injective, then the  sequence 
\[ 0 \to \phi^*E \to \phi^* F \to \phi^* G \to 0 \]
on $Z_0$ is exact.
\end{lemma}
\begin{proof}
Apply the derived pullback functor to the sequence (\ref{seq-lemstaysexact}) to obtain the sequence
\[ 0 \to L^1 \phi^* G \to \phi^*E \to \phi^* F \to \phi^* G \to 0 , \] where $L^1 \phi^* F =0$ since $F$ is a vector bundle.  
Since $\phi^* E \to \phi^* F$ is generically injective, $L^1 \phi^* G$ must be torsion. Since $E$ is a vector bundle, $\phi^*E$ is torsion-free and we conclude that $L^1 \phi^* G = 0$. This proves the lemma.
\end{proof}

We first relate the normal bundles $N_{h_t/X_t}$ and $N_{h/\xX}$.

\begin{lemma}
\label{lem-firstDiagram}
We have  $$N_{h_t/X_t} \cong j_t^* N_{h/\xX}.$$
\end{lemma}
\begin{proof}
Consider the following diagram. 
\catcode`\@=10
\newdimen\cdsep
\cdsep=3em

\def\cdstrut{\vrule height .25\cdsep width 0pt depth .12\cdsep}
\def\@cdstrut{{\advance\cdsep by 2em\cdstrut}}

\def\arrow#1#2{
  \ifx d#1
    \llap{$\scriptstyle#2$}\left\downarrow\cdstrut\right.\@cdstrut\fi
  \ifx u#1
    \llap{$\scriptstyle#2$}\left\uparrow\cdstrut\right.\@cdstrut\fi
  \ifx r#1
    \mathop{\hbox to \cdsep{\rightarrowfill}}\limits^{#2}\fi
  \ifx l#1
    \mathop{\hbox to \cdsep{\leftarrowfill}}\limits^{#2}\fi
}
\catcode`\@=10

\cdsep=3em
$$
\begin{matrix}
& & 0 & & 0 \cr
& & \arrow{u}{} & & \arrow{u}{} \cr
& & \OO_{Y_t}^{N} & \arrow{r}{=} & \OO_{Y_t}^N \cr
& & \arrow{u}{} & & \arrow{u}{} \cr
0 & \arrow{r}{} &  j_t^* T_{\yY}  & \arrow{r}{} & h_t^* i_t^* T_{\xX} & \arrow{r}{} &  j_t^* N_{h/\xX}& \arrow{r}{} & 0          \cr
& &  \arrow{u}{} & & \arrow{u}{} & & \arrow{u}{\cong} \cr
0 & \arrow{r}{} & T_{Y_t} & \arrow{r}{} & h_t^{*}T_{X_t} & \arrow{r}{} &  N_{h_t/X_t} & \arrow{r}{} & 0 \cr
& & \arrow{u}{} & & \arrow{u}{} \cr
& & 0 & & 0 \cr
\end{matrix}
$$
Since $Y_t$ in $\yY$ and $X_t$  in $\xX$ are fibers of fibrations, their normal bundles are trivial of rank $N= \dim S_d$. The first column in the diagram is the definition of the normal bundle $N_{Y_t/\yY}$. The second column in the diagram is the pullback  under $h_t$ of the sequence 
$$0 \rightarrow T _{X_t} \rightarrow T_{\xX}|_{X_t} \rightarrow N_{X_t/\xX} \rightarrow 0$$
defining  $N_{X_t/\xX}$. Since $h_t$ is generically injective and $T_{X_t} \to T_{\xX}|_{X_t}$ is everywhere injective, by Lemma  \ref{lem-staysExact} the pullback of the sequence by $h_t$ remains exact. The top row is an isomorphism of sheaves in the natural way. The bottom row is the sequence defining $N_{h_t/X_t}$. The middle row is the pullback of the sequence defining $N_{h/\xX}$ under $j_t$ using the identification $h \circ j_t = i_t \circ h_t$. It is exact by Lemma \ref{lem-staysExact} and the fact that $Y_t$ passes through a general point of $\yY$. The required isomorphism  $N_{h_t/X_t} \cong j_t^* N_{h/\xX}$ follows from the diagram by the Nine Lemma. 
\end{proof}

We can also express the normal bundle $N_{h_t/X_t}$ in terms of vertical tangent bundles. Let $K$ denote the cokernel of the map $T_{\yY}^{\ver} \to  h^* T_{\xX}^{\ver}$.

\begin{lemma}
 We have 
$$N_{h_t/X_t} \cong j_t^* K.$$ 
\end{lemma}
\begin{proof}
This is another diagram chase. Consider the following diagram.

\catcode`\@=11
\newdimen\cdsep
\cdsep=3em

\def\cdstrut{\vrule height .25\cdsep width 0pt depth .12\cdsep}
\def\@cdstrut{{\advance\cdsep by 2em\cdstrut}}

\def\arrow#1#2{
  \ifx d#1
    \llap{$\scriptstyle#2$}\left\downarrow\cdstrut\right.\@cdstrut\fi
  \ifx u#1
    \llap{$\scriptstyle#2$}\left\uparrow\cdstrut\right.\@cdstrut\fi
  \ifx r#1
    \mathop{\hbox to \cdsep{\rightarrowfill}}\limits^{#2}\fi
  \ifx l#1
    \mathop{\hbox to \cdsep{\leftarrowfill}}\limits^{#2}\fi
}
\catcode`\@=12

\cdsep=3em
$$
\begin{matrix}
& & 0 & & 0 \cr
& & \arrow{u}{} & & \arrow{u}{} \cr
& & j_t^* h^* \pi_2^* T_{\PP^n} & \arrow{r}{=} & h_t^* i_t^* \pi_2^*T_{\PP^n} \cr
& & \arrow{u}{} & & \arrow{u}{} \cr
0 & \arrow{r}{} & j_t^* T_{\yY}  & \arrow{r}{} & h_t^* i_t^*T_{\xX} & \arrow{r}{} & j_t^* N_{h/\xX} & \arrow{r}{} & 0          \cr
& &  \arrow{u}{} & & \arrow{u}{} & & \arrow{u}{\cong} \cr
0 & \arrow{r}{} & j_t^* T_{\yY}^{\ver} & \arrow{r}{} & h_t^{*} i_t^*T_{\xX}^{\ver} & \arrow{r}{} & j_t^* K & \arrow{r}{} & 0 \cr
& & \arrow{u}{} & & \arrow{u}{} \cr
& & 0  & & 0 \cr
\end{matrix}
$$

The first two columns come from the definitions of $T_{\yY}^{\ver}$ and $T_{\xX}^{\ver}$, respectively. The restrictions of each to $Y_t$ remain exact by Lemma \ref{lem-staysExact}. The top isomorphism is also the natural one, since $j_t^* h^* = h_t^* i_t^*$.

The bottom row is exact by Lemma \ref{lem-staysExact}, and the middle row is exact just as in the proof of Lemma \ref{lem-firstDiagram}.

\end{proof}

Thus, in order to study positivity of $N_{h_t/X_t}$, we may study positivity of $K$.

\begin{lemma}[cf \cite{Pacienza2} Section 2]
Let $M_d^{\PP^n}$ be the bundle defined by the sequence 
\[ 0 \to M_d^{\PP^n} \to \OO_{\PP^n} \otimes S_d \to \OO_{\PP^n}(d) \to 0 .\] Then $T_{\xX}^{\ver} \cong \pi_2^*M^{\PP^n}_d$.
\end{lemma}
\begin{proof}
This is yet another diagram chase. Consider the following diagram.

\catcode`\@=11
\newdimen\cdsep
\cdsep=3em

\def\cdstrut{\vrule height .25\cdsep width 0pt depth .12\cdsep}
\def\@cdstrut{{\advance\cdsep by 2em\cdstrut}}

\def\arrow#1#2{
  \ifx d#1
    \llap{$\scriptstyle#2$}\left\downarrow\cdstrut\right.\@cdstrut\fi
  \ifx u#1
    \llap{$\scriptstyle#2$}\left\uparrow\cdstrut\right.\@cdstrut\fi
  \ifx r#1
    \mathop{\hbox to \cdsep{\rightarrowfill}}\limits^{#2}\fi
  \ifx l#1
    \mathop{\hbox to \cdsep{\leftarrowfill}}\limits^{#2}\fi
}
\catcode`\@=12

\cdsep=3em
$$
\begin{matrix}
& & 0 & & 0 \cr
& & \arrow{u}{} & & \arrow{u}{} \cr
& & \pi_2^* \OO(d) & \arrow{r}{=} & \pi_2^* \OO(d) \cr
& & \arrow{u}{} & & \arrow{u}{} \cr
0 & \arrow{r}{} & \OO \otimes S_d  & \arrow{r}{} & \pi_2^*T_{\PP^n} \oplus \OO \otimes S_d & \arrow{r}{} & \pi_2^*T_{\PP^n} & \arrow{r}{} & 0          \cr
& &  \arrow{u}{} & & \arrow{u}{} & & \arrow{u}{\cong} \cr
0 & \arrow{r}{} & T_{\xX}^{\ver} & \arrow{r}{} & T_{\xX} & \arrow{r}{} & \pi_2^*T_{\PP^n} & \arrow{r}{} & 0 \cr
& & \arrow{u}{} & & \arrow{u}{} \cr
& & 0  & & 0 \cr
\end{matrix}
$$

The bottom row is the defining sequence for $T_{\xX}^{\ver}$. The middle row comes from the natural splitting of $\pi_2^*T_{\PP^n} \oplus \OO \otimes S_d$. 

The second column is the normal bundle sequence for $\xX \subset \PP^n \times U$. The maps in the first column are naturally induced by the maps in the second and prove the desired result.

\end{proof}

\begin{lemma}[cf \cite{ClemensRan} Section 3]
\label{lem-sCalc}
There exists an integer $s$ so that there is a surjective map
\[ \bigoplus_{i=1}^s h^* \pi_2^* M_1^{\PP^n} \to K. \]
\end{lemma}
\begin{proof}
Given a degree $d-1$ polynomial $P$, we get a natural map $M_1^{\PP^n} \to M_d^{\PP^n}$, given by multiplication by $P$. Since polynomials of the form $x_i P$ generate the fiber of $M_d^{\PP^n}$ at a point, we see that we can inductively keep adding copies of $M_1^{\PP^n}$ to decrease the cokernel of the combined map.
\end{proof}

In Voisin, Pacienza, and Clemens-Ran \cite{Voisincorrection, Pacienza, Pacienza2, ClemensRan}, the strategy is to show that if $Y_t$ is not contained in the locus on $X_t$ swept out by lines, then $s$ in Lemma \ref{lem-sCalc} is not too large relative to $n$. Since the first Chern class of $M_1^{\PP^n}$ is negative the hyperplane class on $\PP^n$, this bounds the negativity of the normal bundle of any subvariety $Y_t \subset X_t$ that is not contained in the locus swept out by lines.

If $n=3$, we get a map with torsion cokernel if we merely let $s=1$, so bounding $s$ is not useful in studying curves on surfaces in $\PP^3$. Instead, we combine the fact that $M_1^{\PP^n} \cong \Omega_{\PP^n}(1)$ with our understanding of $\Omega_{\PP^n}(1)$ pulled back to $Y_t$ to get better bounds on the degree of $N_{h_t/X_t}$. The statement we use for our purposes is the following.

\begin{corollary}
\label{cor-mapToNormalBundle}
If $n=3$ and $\yY$ is a family of curves, there is a map $h_t^* i_t^* \pi_2^* \Omega_{\PP^n}(1) \to N_{h_t/X_t}$ with torsion cokernel.
\end{corollary}
\begin{proof}
Pull back the surjective map $\bigoplus_{i=1}^s h^* \pi_2^* M_1^{\PP^n} \to K$ from Lemma \ref{lem-sCalc} under $j_t$. Since $K$ has rank $1$ and $N_{h_t/X_t} \cong j_t^* K$, the result follows.
\end{proof}

\section{Scrolls and degree considerations}\label{sec-Proof}
In this section, we prove Theorem \ref{thm-algHyp}. We specialize the discussion in the previous section to the case of curves.

\begin{proposition}
\label{prop-Scrolls}
Let $h: C \to \PP^n$ be a generically injective map of degree $e$ from a smooth curve $C$ to $\PP^n$. Then line bundle quotients of $h^* \Omega_{\PP^n}(1)$ of degree $-m$ give rise to surface scrolls in $\PP^n$ of degree at most $e-m$ that contain $h(C)$.
\end{proposition}
\begin{proof}
The pullback of the Euler sequence on $\PP^n$ to $C$ gives rise to the sequence
\[ 0 \to h^*\Omega_{\PP^n}(1) \to \OO_{C}^{n+1} \to \OO_C(1) \to 0 . \]
Suppose that $Q$ is a degree $-m$ line bundle quotient of $h^*\Omega_{\PP^n}(1)$ $$0 \rightarrow S \rightarrow h^*\Omega_{\PP^n}(1) \rightarrow Q \rightarrow 0$$ with kernel $S$. Then $S$ is a vector sub-bundle of $h^*\Omega_{\PP^n}(1)$ of degree $m-e$. Composing $S$ with the inclusion   $h^*\Omega_{\PP^n}(1) \rightarrow \OO_C^{n+1}$ realizes $S$ as a vector sub-bundle of $\OO_C^{n+1}$, with quotient $Q'$ of rank $2$ and degree $e-m$. By the universal property of projective space, this gives a map from $\psi: \PP(Q') \to \PP^n$ whose image contains $h(C)$. If the map $\psi$ is generically injective, the scroll has degree $e-m$. If $\psi$ is generically $d_0$ to one, then the scroll has degree $\frac{e-m}{d_0}$.
\end{proof}

\begin{corollary}
\label{cor-degreeOfrk1Quot}
Let $h: C \to \PP^n$ be a generically injective map of degree $e$ from a smooth curve $C$ to $\PP^n$. Assume that $h(C)$ does not lie on a surface scroll of degree less than $k$.  Then any rank $1$ quotient of $h^* \Omega_{\PP^n}(1)$ (not necessarily a line bundle) has degree at least $k-e$.
\end{corollary}
\begin{proof}
Let $Q$ be a rank $1$ quotient, and let $L$ be the line bundle given by $Q$ mod torsion. Then $\deg Q \geq \deg L \geq e-k$ by Proposition \ref{prop-Scrolls}.
\end{proof}

\begin{proof}[Proof of Theorem \ref{thm-algHyp}]
Let $X \subset \PP^3$ be a very general degree $d$ surface.  Let $h: C \to X$ be the normalization of a complete intersection curve of type $(d, k)$ in $X$. Assume that $C$ has genus $g$. Then by Corollary \ref{cor-mapToNormalBundle}, there is a map $\alpha: h^* \Omega_{\PP^n}(1) \to N_{h/X}$ with torsion cokernel. Let $Q$ be the image of $\alpha$. Then since $Q$ injects into $N_{h/X}$, $Q$ is a rank 1 quotient of $h^* \Omega_{\PP^n}(1)$ with $\deg Q \leq \deg N_{h/X}$. In order to use Corollary \ref{cor-degreeOfrk1Quot}, we need to understand the smallest possible degree of a scroll containing $h(C)$.

If $k \leq d$, then the degree $k$ hypersurface containing $h(C)$ is the smallest degree surface containing $h(C)$, so any scroll containing $h(C)$ will have degree at least $k$. If $d > k$, then the only irreducible surface of degree less than $k$ containing $h(C)$ is $X$. Since $d \geq 5$ and $X$ is very general, $X$ is irreducible and of general type. In particular, $X$ is not a scroll. Thus, the complete intersection curve $h(C)$ lies on no surface scrolls of degree less than $k$. Hence, by Corollary \ref{cor-degreeOfrk1Quot}, $$\deg N_{h/X} \geq \deg Q \geq k-dk.$$ 

On the other hand, we know that $$\deg N_{h/X} = dk(4-d)+2g-2.$$ Rearranging the inequality
\[ dk(4-d)+2g-2 \geq k-dk ,\]
we obtain
\[ 2g-2 \geq dk(d-5)+k \]
or 
\[ g \geq \frac{dk(d-5)+k}{2}+1 \]
as desired.
\end{proof}

\begin{remark}
If one had a better bound for the minimal degree of the universal line of a scroll containing $h(C)$, one could get better genus bounds. In particular, if the degree were at least $k+s$, then we would get a genus bound of
\[ g \geq \frac{dk(d-5)+k+s}{2} + 1 .\]
One can ask if the cone over $h(C)$ with vertex at the most singular point of $h(C)$ is the scroll containing $h(C)$ with minimal degree of the universal line. If this were true, then the corresponding genus bound would be
\[ g \geq \frac{dk(d-4)-\lfloor \sqrt{dk^2+1} \rfloor +1}{2} .\]
\end{remark}

\bibliographystyle{plain}

\end{document}